\documentclass[letterpaper,10 pt, conference]{ieeeconf}
\IEEEoverridecommandlockouts

\usepackage{amsmath, amssymb}
\usepackage{mathtools}
\usepackage{pgfplots}
\usepackage{tikz}

\DeclareMathOperator{\im}{im}

\usetikzlibrary{arrows,automata, patterns, calc,decorations.pathmorphing,decorations.markings}

\newcommand{\bbR}{\mathbb{R}}

\newcommand{\calD}{\mathcal{D}}

\newcommand{\calR}{\mathcal{R}}
\newcommand{\calX}{\mathcal{X}}
\newcommand{\calV}{\mathcal{V}}

\newcommand{\calY}{\mathcal{Y}}
\newcommand{\calS}{\mathcal{S}}
\newcommand{\calU}{\mathcal{U}}
\newcommand{\calW}{\mathcal{W}}

\newcommand{\btab}{\begin{center}\def\arraystretch{1.5}\begin{tabular}}
		\newcommand{\etab}{\end{tabular}\end{center}}
\newcommand{\bbm}{\begin{bmatrix*}}
	\newcommand{\ebm}{\end{bmatrix*}}
\newcommand{\bvm}{\begin{vmatrix*}}
	\newcommand{\evm}{\end{vmatrix*}}

\newcommand{\set}[2]{\left\{ #1 ~\left|~ \vphantom{#1} #2 \right. \right\}}

\newcommand{\qand}{\quad\text{and}\quad}

\newcommand{\sys}{{\Sigma}}
\newcommand{\ass}{\text{\upshape A}}
\newcommand{\env}{\text{\upshape E}}
\newcommand{\gar}{{\Gamma}}
\newcommand{\con}{\mathcal{C}}
\newcommand{\contract}[1]{\mbox{\ensuremath{\con_{#1} = (\ass_{#1}, \gar_{#1})}}}

\newcommand{\simby}{\preccurlyeq}
\newcommand{\bisim}{\sim}

\newcommand{\meet}{\wedge}

\newcommand{\X}[1]{\calX_{#1}}
\newcommand{\U}[1]{\calU_{#1}}
\newcommand{\Y}[1]{\calY_{#1}}
\newcommand{\D}[1]{\calD_{#1}}
\newcommand{\W}[1]{\calW_{#1}}

\renewcommand{\epsilon}{\varepsilon}

\newtheorem{theorem}{Theorem}
\newtheorem{lemma}[theorem]{Lemma}
\newtheorem{proposition}[theorem]{Proposition}
\newtheorem{remark}{Remark}
\newtheorem{definition}{Definition}
\newtheorem{example}{Example}

\title{Series composition of simulation-based assume-guarantee contracts\\ for linear dynamical systems}
\author{B. M. Shali, H. M. Heidema, A. J. van der Schaft, B. Besselink\thanks{The authors are with the Jan C. Willems Center for Systems and Control, and the Bernoulli Institute for Mathematics, Computer Science, and Artificial Intelligence, University of Groningen, Groningen, The Netherlands; Email: {\emph{b.m.shali@rug.nl}}; {\emph{mariekeheidema@live.nl}}; \emph{a.j.van.der.schaft@rug.nl}; \emph{b.besselink@rug.nl}.}}

\begin{document}
	\maketitle
	\begin{abstract}
		We present assume-guarantee contracts for continuous-time linear dynamical systems with inputs and outputs. These contracts are used to express specifications on the dynamic behaviour of a system. Contrary to existing approaches, we use simulation to compare the dynamic behaviour of two systems. This has the advantage of being supported by efficient numerical algorithms for verification as well as being related to the rich literature on (bi)simulation based techniques for verification and control, such as those based on (discrete) abstractions. Using simulation, we define contract implementation and a notion of contract refinement. We also define a notion of series composition for contracts, which allows us to reason about the series interconnection of systems on the basis of the contracts on its components. Together, the notions of refinement and composition allow contracts to be used for modular design and analysis of interconnected systems.
	\end{abstract}
    \section{Introduction}

	Contract-based design has proven to be an effective method for modular design and analysis of complex interconnected systems \cite{benveniste2018, vincentelli2012, nuzzo2014}. Motivated by this, we present assume-guarantee contracts for continuous-time linear dynamical systems with inputs and outputs in the spirit of \cite{shali2021, shali2021b}. These contracts are defined as a pair of systems called assumptions and guarantees. The assumptions capture the available information about the dynamic behaviour of the environment in which the system is supposed to operate, while the guarantees specify the desired dynamic behaviour of the system when interconnected with a compatible environment. In contrast to \cite{shali2021, shali2021b}, we formalize this with the notion of \emph{simulation}, which is used to compare the dynamic behaviour of two systems.

	Simulation is the one-sided version of the notion of bisimulation, which is used to express (external) system equivalence. Bisimulation finds its origins in the field of computer science, where it was introduced in the context of concurrent processes \cite{milner1995}. In this paper, we adopt the notion of (bi)simulation for continuous-time linear dynamical system introduced in \cite{vanderschaft2004}, see also \cite{vanderschaft2004b, megawati2016}. This notion is very much inspired by the work of Pappas et al. in \cite{pappas2000}, \cite{pappas2002}, \cite{pappas2003}, where the focus is on abstractions, i.e., (bi)similar systems of lower state space dimension.

    Using (bi)simulation as a means of comparing system behaviour has the following advantages. First, as shown in \cite{vanderschaft2004}, \cite{vanderschaft2004b, megawati2016}, efficient numerical procedures for verifying (bi)simulation can be obtained using ideas from geometric control theory \cite{basile1992} and, in particular, the invariant subspace algorithm \cite{trentelman2001}. Second, this connection with geometric control theory allows us to use a multitude of tools in addressing problems relevant to contract-based design, such as constructing  implementations and controllers for implementations. Third, the notion of (bi)simulation has been extended to more general system classes, such as hybrid and transition systems. In fact, there is a rich literature on using discrete abstractions of continuous dynamical systems for the purposes of verification and control \cite{tabuada2009, belta2017}. There are also alternative notions of (bi)simulation, such as approximate (bi)simulation \cite{girard2007}, \cite{girard2011}, and asymptotic (bi)simulation \cite{vinjamoor2010b}. All things considered, using (bi)simulation as a means of comparing system behaviour opens the doors to the vast literature on related research.

    The contributions of this paper are as follows. First, we define contracts and characterize contract implementation as a simulation of one system by another. Second, we define contract refinement, again in terms of simulation, and show that it satisfies properties which allow us to determine if a given contract expresses a stricter specification than another contract. Since simulation can be verified using efficient numerical procedures, it follows that the same procedures can be used to verify contract implementation and refinement. Third, we define the series composition of contracts and show that it satisfies properties which allow us to reason about the series interconnection of two systems on the basis of the contracts that they implement. Together, contract refinement and the series composition of contracts have properties which enable the independent design of components within interconnected systems.

    The contracts in this paper draw inspiration from the contracts introduced in \cite{shali2021, shali2021b}. The main difference with \cite{shali2021,shali2021b}, is that here we use simulation instead of inclusion of external behaviour as a means of comparing system behaviour. As mentioned before, this enables the use of efficient computational tools that are not available for the contracts in \cite{shali2021,shali2021b}. Furthermore, as noticed in the theory of concurrent processes, simulation is more powerful than behavioural inclusion for nondeterministic systems, which will be used throughout this paper.

    Different types of contracts have already been used as specifications for dynamical systems. For example, parametric assume-guarantee contracts are introduced in \cite{kim2017} and used for control synthesis in \cite{khatib2020, chen2021}, while assume-guarantee contracts that can capture invariance are presented in \cite{saoud2021b} and applied in \cite{zonetti2019, loreto2020}. Related work on contracts can also be found in \cite{eqtami2019, ghasemi2020, sharf2021}. Whereas the contracts in this paper express specifications on the dynamics of continuous-time systems, the contracts in \cite{kim2017, khatib2020, chen2021, sharf2021} are defined only for discrete-time systems, and the contracts in \cite{zonetti2019, loreto2020, eqtami2019} cannot express specifications on dynamics. In this respect, the contracts in this paper are most closely related to the contracts in \cite{besselink2019}, while \cite{kerber2009, kerber2010, kerber2011} contain closely related work on compositional reasoning. A key difference with \cite{besselink2019}, however, is that there is no distinction between inputs and outputs and interconnection is defined through variable sharing.

    The remainder of this paper is organized as follows. In Section~\ref{sec:system_class_simulation}, we introduce the classes of systems considered in this paper and develop an appropriate notion of simulation. Then contracts, contract implementation and contract refinement are defined and characterized in Section~\ref{sec:contracts}. Following this, we define and characterize the series composition of contracts in Section~\ref{sec:series_composition}. We finish with concluding remarks in Section~\ref{sec:conclusion}.

    The notation used in this paper is mostly standard. The set of nonnegative real numbers is denoted by $\bbR_{\geq0}$. Finite-dimensional linear (sub)spaces are denoted by capital calligraphic letters. Given a linear subspace $\calV\subset\calX\times\calY$, $\pi_{\calX}(\calV)$ denotes the projection of $\calV$ onto $\calX$, i.e.,
    \begin{equation}
        \pi_{\calX}(\calV) = \set{x\in\calX}{\exists y\in\calY \text{ s.t.\ } (x,y)\in\calV}.
    \end{equation}
    The projection $\pi_{\calY}(\calV)$ is defined similarly. Given a linear map $A:\calX\to\calY$, $\im A$ and $\ker A$ denote the image and kernel of $A$, respectively.

    \section{System classes and simulation}\label{sec:system_class_simulation}
	In this paper, we consider systems of the form
	\begin{equation}\label{eq:sys}
		\sys: \left\lbrace
		\begin{aligned}
			\dot x(t) &= Ax(t) + Bu(t) + Gd(t),\\
			y(t) &= Cx(t),
		\end{aligned}\right.
	\end{equation}
	with state $x(t)\in\X{}$, input $u(t)\in\U{}$, output $y(t)\in\Y{}$, and driving variable $d(t)\in\D{}$. The driving variable $d$ can be used to model disturbance, nondeterminism, unknown inputs, or lack of knowledge about the dynamics of $\sys$. We treat $\sys$ as an open system in which the input $u$ and output $y$ are external variables that interact with the environment, whereas the state $x$ and the driving variable $d$ are internal and do not interact with the environment. As a design goal, we are interested in expressing specifications on the dynamics of the external variables of $\sys$. We will do this with the notion of a contract.

    To define contracts and express specifications, we will make use of systems of the form
	\begin{equation}\label{eq:xi_sys}
		\Xi_i:\left\lbrace
		\begin{aligned}
			\dot x_i(t) &= A_ix(t) + G_id_i(t),\\
			w_i(t) &= C_ix_i(t),\\
			0 &= H_ix_i(t),
		\end{aligned}\right.
	\end{equation}
	with state $x_i(t)\in\X{i}$, output $w(t)\in\W{i}$, and driving variable $d(t)\in\D{i}$.
	The main differences between $\sys$ and $\Xi_i$ is that $\Xi_i$ does not admit an input and includes algebraic constraints. Due to the algebraic constraints, not all initial states lead to feasible trajectories. This motivates the introduction of the \emph{consistent subspace} $\calV_i\subset\X{i}$, defined as the set of initial states $x_i(0)$ for which there exists a driving variable $d_i:\bbR_{\geq0}\to\D{i}$ such that the resulting state trajectory satisfies the algebraic constraint, i.e., $H_ix_i(t) = 0$ for all $t\geq 0$. It can be shown  that $\calV_i$ is the \emph{largest} subspace that satisfies
	\begin{equation}\label{eq:consistent_subspace}
		A_i\calV_i \subset \calV_i + \im G_i \qand \calV_i\subset\ker H_i.
		\vspace{-1mm}
	\end{equation}
	\begin{remark}
		Including algebraic constraints in the systems $\Xi_i$ has two advantages. First, it leads to a more general class of systems, which will allow us to express more general specifications. Second, it allows for easily defining certain interconnections, which will be essential in the definitions of contract refinement and the series composition of contracts.
	\end{remark}

	The theory that we will develop heavily relies on comparing the dynamics of different systems $\Xi_i$. For this, we will make use of the notion of simulation, which itself relies on the notion of simulation relation. The following definition is taken from \cite{besselink2019}, see also \cite{vanderschaft2004b, megawati2016}.

	\begin{definition}
		A linear subspace $\calS\subset\X{1}\times\X{2}$ satisfying $\pi_{\X{i}}(\calS) \subset \calV_i$, $i\in\{1,2\}$, is a \emph{simulation relation} of $\Xi_1$ by $\Xi_2$ if the following implication holds: for all $(x_1(0), x_2(0))\in\calS$ and all $d_1:\bbR_{\geq0}\to\D{1}$ such that $x_1(t)\in\calV_1$ for all $t\geq 0$, there exists $d_2:\bbR_{\geq 0}\to\D{2}$ such that:
		\begin{enumerate}
			\item $\left(x_1(t), x_2(t)\right)\in \calS$ for all $t\geq 0$;
			\item $w_1(t) = w_2(t)$ for all $t\geq0$.
		\end{enumerate}
	\end{definition}

	Using ideas from geometric control theory \cite{basile1992, trentelman2001}, we obtain the following equivalent characterization of a simulation relation based solely on the system matrices, see \cite{besselink2019, vanderschaft2004b, megawati2016} for details.
	\begin{proposition}\label{prop:simulation_relation}
		A linear subspace $\calS\subset\X{1}\times\X{2}$ satisfying $\pi_{\X{i}}(\calS) \subset \calV_i$, $i\in\{1,2\}$, is a simulation relation of $\Xi_1$ by $\Xi_2$ if and only if for all $(x_1,x_2)\in\calS$ and all $d_1\in\D{1}$ such that $A_1x_1 + G_1d_1\in\calV_1$, there exists $d_2\in\D{2}$ such that:
		\begin{enumerate}
			\item $(A_1x_1 + G_1d_1, A_2x_2 + G_2d_2)\in\calS$;
			\item $C_1x_1 = C_2x_2$.
		\end{enumerate}
	\end{proposition}
    \begin{remark}\label{rem:simulation_relation}
        When constructing a simulation relation $\calS$, it is sometimes difficult to ensure that $\pi_{\X{i}}(\calS)\subset\calV_i$. For such cases, note that if $\calS$ satisfies the first condition in Proposition~\ref{prop:simulation_relation}, then
        \begin{equation}
            A_i\pi_{\X{i}}(\calS)\subset\pi_{\X{i}}(\calS) + \im G_i,
        \end{equation}
        hence $\pi_{\X{i}}(\calS)\subset\calV_i$ if and only if $\pi_{\X{i}}(\calS)\subset\ker H_i$.
    \end{remark}

    Simulation is then defined as follows.

    \begin{definition}\label{def:simulation}
        A system $\Xi_1$ is \emph{simulated by} $\Xi_2$, denoted as $\Xi_1\simby\Xi_2$, if there exists a simulation relation $\calS\subset\X{1}\times\X{2}$ of $\Xi_1$ by $\Xi_2$ such that $\pi_{\X{1}}(\calS) = \calV_1$. A simulation relation with this property is called a \emph{full simulation relation}.
    \end{definition}

    If $\Xi_1\simby\Xi_2$, then \emph{any} state trajectory of $\Xi_1$ can be matched by a state trajectory of $\Xi_2$ such that the outputs of $\Xi_1$ and $\Xi_2$ are identical. We can interpret this as $\Xi_2$ having richer dynamics than $\Xi_1$.
    \begin{remark}
        In view of \eqref{eq:consistent_subspace}, computing the consistent subspace of a system $\Xi_i$ amounts to computing the largest $(A_i,G_i)$-invariant subspace contained in $\ker H_i$, which can be done using the invariant subspace algorithm, see \cite{basile1992, trentelman2001} for details. Using the same algorithm, one can compute the largest simulation relation of $\Xi_1$ by $\Xi_2$ and thus determine whether $\Xi_1$ is simulated by $\Xi_2$, see \cite[Theorem~6]{besselink2019} and \cite[Remark~4]{besselink2019} for details. In other words, simulation is supported by efficient numerical procedures for verification.
    \end{remark}
    \begin{remark}
        If $\Xi_1\simby\Xi_2$ and $\Xi_2\simby \Xi_1$, then $\Xi_1$ and $\Xi_2$ can be shown to be \emph{bisimilar}, denoted by $\Xi_1\bisim\Xi_2$. Bisimilarity for systems of the form \eqref{eq:xi_sys} is defined in \cite{megawati2016}, and a proof of this statement is given in Proposition~5.3. Loosely speaking, bisimilar systems have the same external dynamics.
    \end{remark}
    \begin{remark}
        An important property that will be used throughout this paper is that simulation is a \emph{preorder} \cite[Lemma~2]{besselink2019}, i.e., it is \emph{reflexive} ($\Xi_i\simby\Xi_i$ for all $\Xi_i$) and \emph{transitive} ($\Xi_1\simby\Xi_2$ and $\Xi_2\simby\Xi_3$ imply that $\Xi_1\simby\Xi_3$).
    \end{remark}

	\section{Contracts}\label{sec:contracts}

	In this section, we will define contracts, contract implementation, and a notion of contract refinement that will allow us to compare contracts. Consider a system $\sys$ of the form \eqref{eq:sys}. The \emph{environment} $\env$ of $\sys$ is a system of the form
	\begin{equation}
		\env: \left\lbrace
		\begin{aligned}
			\dot x_e &= A_ex_e + G_ed_e,\\
			u &= C_ex_e,\\
			0 &= H_ex_e,
		\end{aligned}\right.
	\end{equation}
	with $x_e\in\X{e}$ and $d_e\in\D{e}$. Here, we have omitted the time variable $t$ for convenience.  The environment $\env$ is interpreted as a system that generates inputs for $\sys$, hence the interconnection of $\env$ and $\sys$ is given by
	\begin{equation}
		\env\meet\sys: \left\lbrace
		\begin{aligned}
			\bbm \dot x_e \\ \dot x \ebm &= \bbm A_e & 0 \\ BC_e & A \ebm\bbm x_e \\ x \ebm + \bbm G_e & 0 \\ 0 & G \ebm \bbm d_e \\ d\ebm,\\
			\bbm u \\ y \ebm &= \bbm C_e & 0 \\ 0 & C \ebm \bbm x_e \\ x \ebm,\\
			0 &= \bbm H_e & 0 \ebm \bbm x_e \\ x \ebm.
		\end{aligned}\right.
	\end{equation}
	which we have obtained by setting the output generated by $\env$ as input to $\sys$, as shown in Figure~\ref{fig:env_meet_sys}.
	\begin{figure}[t]
		\centering
		\begin{tikzpicture}[->,>=stealth',shorten >=1pt,auto,node distance=3cm,
            semithick]
            \tikzset{box/.style = {shape = rectangle,
                    color=black,
                    fill=white!96!black,
                    text = black,
                    inner sep = 1mm,
                    minimum width = 12mm,
                    minimum height = 6mm,
                    draw}
            }

			\node[box] (S1) at (0,0) {$\env$};
			\node[box] (S2) at (2,0) {$\sys$};

			\draw (S1) -- (S2);
			\draw (1, 0) -- (1, 0.5) -- (2.5, 0.5) --node[pos=0.7, above] {$u$} (4, 0.5);
			\draw (S2) -- node[pos=0.7, below] {$y$} (4, 0);

            \draw[-,dashed] (-1,0.75) -- (3,0.75) -- (3,-0.625) -- (-1,-0.625) -- (-1,0.75);
		\end{tikzpicture}
		\caption{The interconnection $\env\meet\sys$.}
		\label{fig:env_meet_sys}
		\vspace{-5mm}
	\end{figure}
	We are interested in specifying the dynamic behaviour of $\env\meet\sys$ only for relevant environments $\env$. This will be formalized with the notion of a contract, which will require the definition of another two systems. First, the \emph{assumptions} $\ass$ are a system of the form
	\begin{equation}
		\ass: \left\lbrace
		\begin{aligned}
			\dot x_{a} &= A_ax_a + G_ad_a,\\
			u &= C_ax_a,\\
			0 &= H_ax_a,
		\end{aligned}\right.
	\end{equation}
	with $x_a\in\X{a}$ and $d_a\in\D{a}$. Assumptions have the same form as environments and they can be compared using simulation. Second, the \emph{guarantees} $\gar$ are a system of the form
	\begin{equation}
		\gar: \left\lbrace
		\begin{aligned}
			\dot x_g &= A_gx_g + G_gd_g,\\
			\bbm u \\ y \ebm &= \bbm C^u_g \\ C^y_g \ebm x_g,\\
			0 &= H_g x_g,
		\end{aligned}\right.
	\end{equation}
	with $x_g\in\X{g}$ and $d_g\in\D{g}$. Guarantees have the same form as the interconnection $\env\meet\sys$ and they can be compared using simulation. With assumptions and guarantees defined, we are ready to define contracts.

	\begin{definition}
		A contract $\contract{}$ is a pair of assumptions and guarantees.
	\end{definition}

	A contract is used as a specification in the following sense.
	\begin{definition}
		Consider a contract $\contract{}$. An environment $\env$ is \emph{compatible} with $\con$ if
		\begin{equation}
			\env\simby\ass.
		\end{equation}
		A system $\sys$ \emph{implements} $\con$ if
		\begin{equation}\label{eq:implementation_def}
			\env\meet\sys\simby\gar
		\end{equation}
		for any environment $\env$ compatible with $\con$.
	\end{definition}

	In other words, the assumptions capture the available information about the dynamics of the environments in which our system is supposed to operate, thus leading to a class of compatible environments, while the guarantees specify the desired dynamics of our system when interconnected with a compatible environment, thus leading to a class of implementations.

	We can check if a given system $\sys$ implements a given contract $\contract{}$ without having to construct all compatible environments. To show this, we will make use of the following lemma.
    \begin{lemma}\label{lem:implementation}
        If $\env\simby\ass$, then $\env\meet\sys\simby\ass\meet\sys$.
    \end{lemma}
    \begin{proof}
        Let $\calS_e$ be a full simulation relation of $\env$ by $\ass$. We will show that the subspace $\calS\subset(\X{e}\times\X{})\times(\X{a}\times\X{})$ defined by
        \begin{equation}
            \calS = \set{(x_e,x,x_a,x)}{(x_e,x_a)\in\calS_e}
        \end{equation}
        is a full simulation relation of $\env\meet\sys$ by $\ass\meet\sys$. First, note that the consistent subspace of $\env\meet\sys$ is given by $\calV_e\times \X{}$, and the consistent subspace of $\ass \meet\sys$ is given by $\calV_a\times\X{}$, where $\calV_e$ and $\calV_a$ are the consistent subspaces of $\env$ and $\ass$, respectively. Since $\pi_{\X{e}}(\calS_e) = \calV_e$ and $\pi_{\X{a}}(\calS_e) \subset \calV_a$, it follows that
        \begin{equation}
            \pi_{\X{e}\times\X{}} (\calS) = \calV_{e}\times\X{} \qand \pi_{\X{a}\times\X{}}(\calS) \subset \calV_{a}\times\X{}.
        \end{equation}
        Let $(x_e,x,x_a,x)\in\calS$ and take $d_e\in\calD_e$, $d\in\calD$, such that
        \begin{equation}
            (A_ex_e + G_ed_e, BC_ex_e + Ax + Gd)\in\calV_e\times \X{}.
        \end{equation}
        For later reference, let $s = BC_ex_e + Ax + Gd$. As $(x_e,x_a)\in\calS_e$ and $A_ex_e + G_ed_e\in\calV_e$, it follows from Proposition 1 that there  there exists $d_a\in\calD_a$ such that
        \begin{align}
            (A_ex_e + G_ed_e, A_ax_a + G_ad_a) &\in \calS_e \label{eq:Se_invariance},\\
            C_ex_e &= C_ax_a \label{eq:Se_equality}.
        \end{align}
        Then \eqref{eq:Se_invariance} implies that
        \begin{equation}\label{eq:env_meet_sys_simby_env_meet_ass_cond1}
            (A_ex_e + G_ed_e, s, A_ax_a + G_ad_a, s) \in \calS,
        \end{equation}
        while \eqref{eq:Se_equality} implies that $s = BC_ax_a + Ax + Gd$ and
        \begin{equation}\label{eq:env_meet_sys_simby_env_meet_ass_cond2}
            \bbm C_e & 0 \\ 0 & C \ebm\bbm x_e \\ x \ebm = \bbm C_a & 0 \\ 0 & C \ebm \bbm x_a \\ x \ebm.
        \end{equation}
        Using Proposition~\ref{prop:simulation_relation}, we conclude that $\calS$ is a full simulation relation of $\env\meet\sys$ by $\ass\meet\sys$ and thus $\env\meet\sys\simby\ass\meet\sys$.
    \end{proof}

	As an almost immediate consequence of Lemma~\ref{lem:implementation}, we obtain the following necessary and sufficient condition for contract implementation.
	\begin{theorem}\label{thm:implementation}
		A system $\sys$ implements the contract $\contract{}$ if and only if
		\begin{equation}\label{eq:implementation_thm}
			\ass\meet\sys\simby\gar.\\[2mm]
		\end{equation}
	\end{theorem}
	\begin{proof}
		Suppose that $\sys$ implements $\con$. Since $\ass$ is an environment compatible with $\con$, it follows that \eqref{eq:implementation_thm} holds. Conversely, suppose that \eqref{eq:implementation_thm} holds and let $\env$ be compatible with $\con$, that is, $\env \simby\ass$. In view of Lemma~\ref{lem:implementation}, we have that $\env\meet\sys\simby\ass\meet\sys$, hence $\env\meet\sys\simby\gar$ because simulation is transitive. Since $\env\meet\sys\simby\gar$ for any $\env$ compatible with $\con$, we conclude that $\sys$ implements $\con$.
	\end{proof}

    \begin{remark}\label{rem:implementation}
        Clearly, two contracts define the same class of compatible environments if and only if their assumptions are bisimilar. However, two contracts can define the same class of implementations even if their guarantees are not bisimilar. For instance, $\contract{}$ defines the same class of implementations as $\con' = (\ass,\ass\meet\gar)$, where $\ass\meet\gar$ is obtained by equating the outputs $u$ of $\ass$ and $\gar$, that is,
        \begin{equation}\label{eq:ass_meet_gar}
            \ass\meet\gar:\left\lbrace
            \begin{aligned}
                \bbm \dot x_a \\ \dot x_g \ebm &= \bbm A_a & 0 \\ 0 & A_g \ebm \bbm x_a \\ x_g \ebm + \bbm G_a & 0 \\ 0 & G_g \ebm\bbm d_a\\ d_g \ebm,\\
                \bbm u \\ y \ebm &= \bbm C_a & 0 \\ 0 & C^y_{g} \ebm \bbm x_a \\ x_g \ebm,\\
                0 &= \bbm H_a & 0 \\ 0 & H_g \\ C_a & - C^u_g \ebm.
            \end{aligned}\right.\hspace{-3mm}
        \end{equation}
        Indeed, if $\calS$ is a full simulation relation of $\ass\meet\sys$ by $\gar$, then
        \begin{equation}
            \calS' = \set{(x_{a},x,x_{a},x_g)}{(x_a,x,x_g)\in\calS}
        \end{equation}
        is a full simulation relation of $\ass\meet\sys$ by $\ass\meet\gar$, hence, due to Theorem~\ref{thm:implementation}, every implementation of $\con$ is also an implementation of $\con'$. Conversely, it can be shown that $\ass\meet\gar\simby\gar$ with a full simulation relation given by
        \begin{equation}
            \calS = \set{(x_a,x_g, x_g)}{(x_a,x_g)\in\calV_{a\meet g}},
        \end{equation}
        where $\calV_{a\meet g}$ is the consistent subspace of $\ass\meet\gar$. Therefore, due to Theorem~\ref{thm:implementation} and the transitivity of simulation, every implementation of $\con'$ is also an implementation of $\con$.
    \end{remark}
    \begin{remark}
        A contract $\contract{}$ is \emph{consistent} if it can be implemented. Not every contract is consistent. To see this, note that $u$ is an input in $\sys$, hence any restriction on the dynamics of $u$ in $\ass\meet\sys$ come from the assumptions $\ass$. Therefore, $\ass\meet\sys\simby\gar$ only if any restrictions on the dynamics of $u$ in $\gar$ are already present in $\ass$, that is, $\ass\simby\gar^u$, where $\gar^u$ is obtained from $\gar$ by considering only $u$ as an output. Indeed, if $\calS$ is a full simulation relation of $\ass\meet\sys$ by $\gar$, then it can be shown that $\pi_{\X{a}\times\X{g}}(\calS)$ is a full simulation relation of $\ass$ by $\gar^u$. Consequently, the condition $\ass\simby\gar^u$ is necessary (but not sufficient) for consistency.
    \end{remark}

    Next, we define the notion of refinement, which allows us to compare two contracts.
    \begin{definition}\label{def:refinement}
    	A contract $\contract{1}$ refines another contract $\contract{2}$, denoted as $\con_1\simby\con_2$, if
    	\begin{equation}\label{eq:refinement}
    		\ass_2\simby\ass_1 \qand \ass_2\meet\gar_1\simby \gar_2\\[2mm]
    	\end{equation}
    \end{definition}

	Refinement allows us to determine if a contract expresses a stricter specification than another contract. In particular, the following theorem shows that if $\con_1\simby\con_2$, then $\con_1$ defines a larger class of compatible environments but a smaller class of implementations than $\con_2$.
    \begin{theorem}\label{thm:refinement}
        If $\con_1\simby\con_2$, then the following hold:
        \begin{enumerate}
            \item any environment compatible with $\con_2$ is compatible with $\con_1$;
            \item any implementation of $\con_1$ is an implementation of $\con_2$.\vspace{1mm}
        \end{enumerate}
    \end{theorem}
	\begin{proof}
		Let $\contract{1}$ and $\contract{2}$, and suppose that $\con_1\simby\con_2$, that is, \eqref{eq:refinement} holds. Let $\env$ be an environment compatible with $\con_2$, that is, $\env\simby\ass_2$. Since $\ass_2\simby\ass_1$ and simulation is transitive, it follows that $\env\simby\ass_1$ and thus $\env$ is also compatible with $\con_1$. Next, suppose that $\sys$ implements $\con_1$. Note that $\ass_2$ is an environment compatible with $\con_1$ because $\ass_2\simby\ass_1$. Consequently, we must have that $\ass_2\meet\sys\simby\gar_1$. As explained in Remark~\ref{rem:implementation}, this implies that $\ass_2\meet\sys\simby\ass_2\meet\gar_1$, hence $\ass_2\meet\sys\simby\gar_2$ because $\ass_2\meet\gar_1\simby\gar_2$ and simulation is transitive. Using Theorem~\ref{thm:implementation}, we conclude that $\sys$ also implements $\con_2$.
	\end{proof}

	Intuitively, Theorem~\ref{thm:refinement} tells us that $\con_1\simby\con_2$ only if $\con_1$ has stricter guarantees than $\con_2$ that have to be met in the presence of weaker assumptions than those of $\con_2$. In other words, $\con_1$ expresses a stricter specification than $\con_2$.

	\begin{remark}
		It is not clear whether the converse of Theorem~\ref{thm:refinement} is true. Although any environment compatible with $\contract{2}$ is compatible with $\contract{1}$ if and only if $\ass_2\simby\ass_1$, it is possible that any implementation of $\con_1$ is an implementation of $\con_2$ even if $\ass_2\meet\gar_1\simby\gar_2$ does not hold. Nevertheless, we know that the converse of Theorem~\ref{thm:refinement} is true when simulation is replaced by inclusion of external behaviour, as shown in \cite{shali2021b}. Behavioural inclusion and simulation are closely related (they are equivalent for deterministic systems), which suggests that condition \eqref{eq:refinement} is ``close" to being equivalent to the conditions in Theorem~\ref{thm:refinement}. 
	\end{remark}

    \section{Series composition of contracts}\label{sec:series_composition}

    In this section, we will define the series composition of two contracts and will show that it satisfies desirable properties for modular analysis. The series composition of contracts can be used to reason about the series interconnection of systems on the basis of the contracts on its components. Loosely speaking, we want the series composition of two contracts to be implemented by the series interconnection of \emph{any} of their implementations. To make this precise, we first define the series interconnection of systems of the form \eqref{eq:sys}.
    \begin{definition}
        Consider systems $\sys_1$ and $\sys_2$ of the form~\eqref{eq:sys}. The \emph{series interconnection} of $\sys_1$ to $\sys_2$, denoted as ${\sys_1\to\sys_2}$, is obtained by setting the output of $\sys_1$ as input of $\sys_2$, as shown in Figure~\ref{fig:sys_series_interconnection}. In other words, the series interconnection $\sys_1\to\sys_2$ is given by
        \begin{equation}
            \sys_1\to\sys_2: \left\lbrace
            \begin{aligned}
                \bbm \dot x_{1} \\ \dot x_{2} \ebm &= \bbm A_{1} & 0 \\ B_2C_1 & A_{2} \ebm \bbm x_{1} \\ x_{2} \ebm + \bbm B_1 \\ 0 \ebm u\\
                &\qquad + \bbm G_{1} & 0 \\ 0 & G_{2} \ebm \bbm d_{1} \\ d_{2} \ebm,\\[1mm]
                y &= \bbm 0 & C_2 \ebm \bbm x_{1} \\ x_{2} \ebm.
            \end{aligned}\right.\\[2mm]
        \end{equation}
    \end{definition}
    \begin{figure}
        \centering
        \begin{tikzpicture}[->,>=stealth',shorten >=1pt,auto,node distance=3cm,
            semithick]
            \tikzset{box/.style = {shape = rectangle,
                    color=black,
                    fill=white!96!black,
                    text = black,
                    inner sep = 1mm,
                    minimum width = 12mm,
                    minimum height = 6mm,
                    draw}
            }
            \node[box] (S1) at (0,0) {$\sys_1$};
            \node[box] (S2) at (2.75,0) {$\sys_2$};

            \draw (S1) -- node[pos=0.5, above] {$y_{1} = u_2$} (S2);
            \draw (-2.5,0) -- node[pos=0.25, above] {$u_{}$} node[pos=0.75, above] {$u_1$} (S1);
            \draw (S2) -- node[pos=0.25, above] {$y_{2}$} node[pos=0.75, above] {$y_{}$} (5.25,0);

            \draw[-,dashed] (-1.5,0.625) -- (4.25,0.625) -- (4.25,-0.625) -- (-1.5,-0.625) -- (-1.5,0.625);
        \end{tikzpicture}
        \caption{The series interconnection $\sys_1\to\sys_2$.}
        \label{fig:sys_series_interconnection}
    \end{figure}

    Therefore, given contracts $\con_1$ and $\con_2$ for $\sys_1$ and $\sys_2$, respectively, our goal is to define a contract $\con_1\to\con_2$ which $\sys_1\to\sys_2$ is guaranteed to implement. This will naturally lead us to consider the series interconnection of guarantees, defined below.
    \begin{definition}\label{def:gar_series_interconnection}
        Consider guarantees $\gar_1$ and $\gar_2$. The \emph{series interconnection} of $\gar_1$ to $\gar_2$, denoted as $\gar_1\to\gar_2$, is obtained by setting the output $y_1$ of $\gar_1$ equal to the output $u_2$ of $\gar_2$, as shown in Figure~\ref{fig:gar_series_interconnection}. In other words, the series interconnection $\gar_1\to\gar_2$ is given by
        \begin{equation}
        \gar_1\to\gar_2: \left\lbrace
        \begin{aligned}
            \bbm \dot x_{g_1} \\ \dot x_{g_2} \ebm &= \bbm A_{g_1} & 0 \\ 0 & A_{g_2} \ebm \bbm x_{g_1} \\ x_{g_2} \ebm \\ &\qquad+ \bbm G_{g_1} & 0 \\ 0 & G_{g_2} \ebm \bbm d_{g_1} \\ d_{g_2} \ebm,\\[1mm]
            \bbm u \\ y \ebm &= \bbm C^u_{g_1} & 0 \\ 0 & C^y_{g_2} \ebm \bbm x_{g_1} \\ x_{g_2} \ebm\\[1mm]
            0 &= \bbm H_{g_1} & 0 \\ 0 & H_{g_2} \\ C^y_{g_1} & - C^u_{g_2} \ebm \bbm x_{g_1} \\ x_{g_2} \ebm.
        \end{aligned}\right.\\[2mm]
    \end{equation}
    \end{definition}

    \begin{figure}
        \centering
        \begin{tikzpicture}[->,>=stealth',shorten >=1pt,auto,node distance=3cm,
            semithick]
            \tikzset{box/.style = {shape = rectangle,
                    color=black,
                    fill=white!96!black,
                    text = black,
                    inner sep = 1mm,
                    minimum width = 12mm,
                    minimum height = 6mm,
                    draw}
            }
            \node[box] (G1) at (0,0) {$\gar_1$};
            \node[box] (G2) at (0,-1.25) {$\gar_2$};
            \node[outer sep=0pt, inner sep=0pt] (equality) at (1.5,-.625) {$\scalebox{3}[1]{=}$};

            \draw ([yshift=1mm] G1.east) --node[pos=0.8, above] {$u_{}$} node[pos=0.15, above] {$u_1$} (3,0.1);
            \draw ([yshift=-1mm] G2.east) --node[pos=0.8, below] {$y^{}$} node[pos=0.15, below] {$y_2^{}$} (3,-1.35);

            \draw ([yshift=-1mm] G1.east) --node[pos=0.4, below] {$y_1$}  (1.5,-0.1) -- (equality);
            \draw ([yshift=1mm] G2.east) --node[pos=0.4, above] {$u_2$} (1.5,-1.15) -- ([yshift=1pt] equality.south);

%
            \draw[-,dashed] (-1,0.625) -- (2,0.625) -- (2,-1.875) -- (-1,-1.875) -- (-1,0.625);
        \end{tikzpicture}
        \caption{The series interconnection $\gar_1\to\gar_2$.}
        \label{fig:gar_series_interconnection}
        \vspace{-5mm}
    \end{figure}

    Since $\sys_1$ and $\sys_2$ are designed to work only in interconnection with environments compatible with $\con_1$ and $\con_2$, respectively, it is only natural to require that, for any environment $\env$ compatible with $\con_1\to\con_2$, the environments of $\sys_1$ and $\sys_2$ in the interconnection $\env\meet(\sys_1\to\sys_2)$ are compatible with $\con_1$ and $\con_2$, respectively. Consequently, since the environment of $\sys_1$ in $\env\meet(\sys_1\to\sys_2)$ is $\env$ itself, it immediately follows that we must have $\env\simby\ass_1$. On the other hand, since the environment of $\sys_2$ in $\env\meet(\sys_1\to\sys_2)$ is $(\env\meet\sys_1)^y$, where $(\env\meet\sys_1)^y$ is obtained from $\env\meet\sys_1$ by considering only $y_1$ as an output, it follows that we must also have $(\env\meet\sys_1)^y\simby\ass_2$. With this in mind, consider the following definition.
    \begin{definition}\label{def:series_composition}
    	Consider contracts $\contract{1}$ and $\contract{2}$. We say that $\con_1$ is \emph{series composable} to $\con_2$ if
    	\begin{equation}
    		(\ass_1\meet\gar_1)^y\simby\ass_2.
    	\end{equation}
    	In this case, the \emph{series composition} of $\con_1$ to $\con_2$, denoted by $\con_1\to\con_2$, is defined as
        \begin{equation}
            \con_1\to\con_2 = (\ass_1, \gar_1\to\gar_2).\\[2mm]
        \end{equation}
    \end{definition}

    The following theorem, whose proof can be found in the appendix, shows that series composition satisfies the properties mentioned above.
    \begin{theorem}\label{thm:series_composition}
        Consider contracts $\con_1$ and $\con_2$ such that $\con_1$ is series composable to $\con_2$. If $\sys_1$ and $\sys_2$ implement $\con_1$ and $\con_2$, respectively, and $\env$ is compatible with $\con_1\to\con_2$, then the following conditions hold:
        \begin{enumerate}
            \item the environment of $\sys_1$ in $\env\meet(\sys_1\to\sys_2)$ is compatible with $\con_1$;
            \item the environment of $\sys_2$ in $\env\meet(\sys_1\to\sys_2)$ is compatible with $\con_2$;
            \item $\sys_1\to\sys_2$ implements $\con_1\to\con_2$.\vspace{1mm}
        \end{enumerate}
    \end{theorem}

	\begin{remark}
		Contract refinement and the series composition of contracts have properties that enable the independent design of components within interconnected systems. As a simple example, suppose that we want to design $\sys_1$ and $\sys_2$ such that the series interconnection $\sys_1\to\sys_2$ implements an overall contract $\con$. Using the definition of the series composition, we can construct contracts $\con_1$ and $\con_2$ such that $\con_1$ is series composable to $\con_2$ and $\con_1\to\con_2\simby\con$. Consequently, if $\sys_1$ implements $\con_1$ and $\sys_2$ implements $\con_2$, then, due to Theorem~\ref{thm:series_composition}, we know that $\sys_1\to\sys_2$ implements $\con_1\to\con_2$, and thus, due to Theorem~\ref{thm:refinement}, $\sys_1\to\sys_2$ implements $\con$. This means that the designer of $\sys_1$ need only implement $\con_1$ and need not concern themselves with the design of $\sys_2$ or the integration of $\sys_1$ into the series interconnection $\sys_1\to\sys_2$. In other words, $\sys_1$ can be designed independently of $\sys_2$. The same is true for $\sys_2$, of course.
	\end{remark}

    \begin{remark}
        It can be shown that the series composition has the following property in relation to refinement. Suppose that $\con_1$ is series composable to $\con_2$ and $\con_1'$ is series composable to $\con_2'$. If $\con_1'\simby\con_1$ and $\con_2'\simby\con_2$, then
        \begin{equation}
            \con_1'\to\con_2'\simby\con_1\to\con_2
        \end{equation}
        The proof of this statement is rather long and technical, and is thus beyond the scope of this paper.
    \end{remark}

	We conclude this section with a simple academic example of series composability and the series composition.
		\begin{example}
			Consider the contract $\contract{}$ with
			\begin{equation}
				\ass: \left\lbrace
				\begin{aligned}
					\dot x_a = d_a,\\
					u = x_a,
				\end{aligned}\right.
				\quad\
				\gar: \left\lbrace
				\begin{aligned}
					\dot x_g &= \bbm 0 & 0 \\ I & 0 \ebm x_g + \bbm I \\ 0 \ebm d_g,\\
					\bbm u \\ y \ebm &= \bbm I & 0 \\ 0 & I \ebm x_g.
				\end{aligned}\right.\hspace{-2mm}\vspace{2mm}
			\end{equation}
		\end{example}
		Note that $u$ is essentially free in $\ass$ since the only restriction is that $\dot u = d_a$ for some $d_a:\bbR_\geq \to \calD_a$. Similarly, $u$ is essentially free in $\gar$, whereas $y$ is such that $\dot y = u$, that is, $\gar$ represents a single integrator. Since $u$ is essentially free in $\ass$, we expect that $(\ass\meet\gar)^y\simby \ass$. Indeed, we can show that the subspace $\calS\subset(\X{a}\times\X{g})\times\X{a}$ defined by
		\begin{equation*}
			\calS = \set{(x_a, x_g, x_{a}')}{(x_a,x_g)\in\calV_{a\meet g},\ x_a' = \bbm 0 & I \ebm x_g}
		\end{equation*}
		is a full simulation relation of $(\ass\meet\gar)^y$ by $\ass$, where $\ass\meet\gar$ is given in \eqref{eq:ass_meet_gar}. To do this, let $(x_a, x_g, x_a')\in \calS$ and take $d_a\in\D{a}$ and $d_g\in\D{g}$ such that
		\begin{equation}
			\left( d_a, \bbm 0 & 0 \\ I & 0 \ebm x_g + \bbm I \\ 0 \ebm d_g\right) \in \calV_{a\meet g}.
		\end{equation}
		Note that $d_a' = \bbm I & 0 \ebm x_g$ is such that
		\begin{equation}
			\left( d_a, \bbm 0 & 0 \\ I & 0 \ebm x_g + \bbm I \\ 0 \ebm d_g, d_a'\right) \in \calS.
		\end{equation}
		Furthermore, since $x_a' = \bbm 0 & I \ebm x_g$, it follows that
		\begin{equation}
			\bbm 0 & 0 & I \ebm \bbm x_a \\ x_g \ebm = x_a',
		\end{equation}
		hence $\calS$ is a simulation relation of $(\ass\meet\gar)^y$ by $\ass$ due to Proposition~\ref{prop:simulation_relation}. As $\pi_{\X{a}\times\X{g}}(\calS) = \calV_{a\meet g}$, it follows that $\calS$ is a full simulation relation and $(\ass\meet\gar)^y\simby \ass$. This means that $\con$ is series composable to $\con$ and $\con\to\con = (\ass,\gar\to\gar).$ Note that, by Definition~\ref{def:gar_series_interconnection}, $\gar\to\gar$ is given by
		\begin{equation}
			\gar\to\gar: \left\lbrace
			\begin{aligned}
				\dot x_g &= \bbm 0 & 0 & 0 & 0\\ I & 0& 0 & 0 \\ 0 & 0 & 0 & 0 \\ 0 & 0 & I & 0 \ebm x_g + \bbm I & 0 \\ 0 & 0 \\ 0 & I \\ 0 & 0 \ebm d_g,\\
				\bbm u \\ y \ebm &= \bbm I & 0 & 0 & 0 \\ 0 & 0 & 0 & I \ebm x_g,\\
				0 & = \bbm 0 & I & -I & 0 \ebm x_g
			\end{aligned}\right.
		\end{equation}
		and we have that $\ddot y = u$, that is, $\gar\to\gar$ represents a double integrator, as expected.

	\section{Conclusion}\label{sec:conclusion}
	
	We presented assume-guarantee contracts for linear dynamical systems with inputs and outputs. In particular, we defined contracts as a pair of linear dynamical systems called assumptions and guarantees. We defined contract implementation using the notion of simulation. We also defined and characterized notions of contract refinement and the series composition of contracts. All relevant conditions are in terms of simulation and can be verified using the efficient numerical algorithm for verifying simulation.

    Future work will focus on defining different types of contract composition (e.g., feedback) in order to reason about more general system interconnections, and on addressing the problems of constructing implementations and synthesizing controllers for implementations.

    \appendix
    \hspace{1em}\emph{Proof of Theorem~\ref{thm:series_composition}:}
    To begin with, let $\contract{1}$ and $\contract{2}$. Note that the environments of $\sys_1$ and $\sys_2$ in $\env\meet(\sys_1\to\sys_2)$ are given by $\env$ and $(\env\meet\sys_1)^y$, respectively. Therefore, the first two conditions can be rewritten as:
    \begin{enumerate}
        \item $\env\simby\ass_1$;
        \item $(\env\meet\sys_1)^y\simby\ass_2$.
    \end{enumerate}
    The first condition holds because $\env$ is compatible with
    \begin{equation}
        \con_1\to\con_2 = (\ass_1, \gar_1\to\gar_2).
    \end{equation}
    Then, due to Lemma~\ref{lem:implementation}, it follows that
    \begin{equation}
        (\env\meet\sys_1)^y\simby(\ass_1\meet\sys_1)^y.
    \end{equation}
    We have that $\ass_1\meet\sys_1\simby\gar_1$ because $\sys_1$ implements $\con_1$. As explained in Remark~\ref{rem:implementation}, this implies that $\ass_1\meet\sys_1\simby\ass_1\meet\gar_1$ and, in particular, that
    \begin{equation}\label{eq:simulation_S12}
        (\ass\meet\sys_1)^y \simby (\ass_1\meet\gar_1)^y.
    \end{equation}
    We also have that $(\ass_1\meet\gar_1)^y\simby\ass_2$ because $\con_1$ is series composable to $\con_2$. Since simulation is transitive, it follows that the second condition also holds.

    Due to Theorem~\ref{thm:implementation}, the third condition is equivalent to:
    \begin{itemize}
        \item[3)] $\ass_1\meet(\sys_1\to\sys_2)\simby\gar_1\to\gar_2$.
    \end{itemize}
    We will show that the third condition holds by constructing a full simulation relation of $\ass_1\meet(\sys_1\to\sys_2)$ by $\gar_1\to\gar_2$. To this end, let $\calS_1$ be a full simulation relation of $\ass_1\meet\sys_1$ by $\gar_1$, $\calS_2$ be a full simulation relation of $\ass_2\meet\sys_2$ by $\gar_2$, and $\calR$ be a full simulation relation of $(\ass_1\meet\sys_1)^y$ by $\ass_2$. Note that $\calS_1$ and $\calS_2$ exist because $\sys_1$ implements $\con_1$ and $\sys_2$ implements $\con_2$. On the other hand, $\calR$ exists because the second condition holds and $\ass_1$ is compatible with $\con_1\to\con_2$. We claim that the subspace $\calS\subset(\X{a}\times\X{1}\times\X{2})\times(\X{g_1}\times\X{g_2})$ defined by
    \begin{equation}
        \calS = \set{(x_{a_1}, x_1, x_2, x_{g_1}, x_{g_2})}{
            \begin{aligned}
                (x_{a_1}, x_1, x_{g_1}) &\in \calS_1\\
                (x_{a_1}, x_1, x_{a_2}) &\in \calR\\
                (x_{a_2}, x_2, x_{g_2}) &\in \calS_{2}
        \end{aligned}}
    \end{equation}
    is a full simulation relation of $\ass_1\meet(\sys_1\to\sys_2)$ by $\gar_1\to\gar_2$.

    To show this, we first note that the consistent subspace of $\ass\meet(\sys_1\to\sys_2)$ is $\calV_{a}\times\X{1}\times\X{2}$. With this in mind, let $(x_{a_1}, x_1, x_2, x_{g_1}, x_{g_2})\in\calS$ and take $d_{a_1}\in\D{a_1}$, $d_1\in\D{1}$ and $d_2\in\D{2}$ such that $A_{a_1}x_{a_1} + G_{a_1}d_{a_1}\in\calV_{a_1}$. Then $(x_{a_1}, x_1, x_{g_1}) \in \calS_1$ and $A_{a_1}x_{a_1} + G_{a_1}d_{a_1}\in\calV_{a_1}$, hence, due to Proposition~\ref{prop:simulation_relation}, there exists $d_{g_1}\in\D{g_1}$ such that
    \begin{align}\label{eq:S1_invariance}
        \bbm A_{a_1}x_{a_1} + G_{a_1} d_{a_1} \\ B_1C_{a_1}x_{a_1} + A_1x_1 + G_1d_1 \\A_{g_1}x_{g_1} + G_{g_1}d_{g_1} \ebm &\in \calS_1,\\[1mm]
        \label{eq:S1_equalities}
        \bbm C_{a_1} & 0 \\ 0 & C_1 \ebm \bbm x_{a_1} \\ x_1 \ebm &= \bbm C^u_{g_1} \\  C^y_{g_1} \ebm x_{g_1}.
    \end{align}
    By definition of $\calS$, there exists $x_{a_2}\in\X{a_2}$ such that $(x_{a_1}, x_1, x_{a_2}) \in \calR$. As $A_{a_1}x_{a_1} + G_{a_1}d_{a_1}\in\calV_{a_1}$, it follows from Proposition~\ref{prop:simulation_relation} that there exists $d_{a_2}\in\D{a_2}$ such that
    \begin{align}\label{eq:S12_invariance}
        \bbm A_{a_1}x_{a_1} + G_{a_1} d_{a_1} \\ B_1C_{a_1}x_{a_1} + A_1x_1 + G_1d_1 \\ A_{a_2}x_{a_2} + G_{a_2}d_{a_2} \ebm &\in\calR,\\[1mm]
        \label{eq:S12_equalities}
        \bbm 0 & C_1 \ebm \bbm x_{a_1} \\ x_1 \ebm &= C_{a_2}x_{a_2}.
    \end{align}
    Note that $\pi_{\X{a_2}}(\calR) \subset \calV_{a_2}$,
    hence $A_{a_2}x_{a_2} + G_{a_2}d_{a_2}\in \calV_{a_2}$. Consequently, since $(x_{a_2}, x_2, x_{g_2}) \in \calS_{2}$, it follows from Proposition~\ref{prop:simulation_relation} that there exists $d_{g_2}\in\D{g_2}$ such that
    \begin{align}
        \bbm A_{a_2}x_{a_2} + G_{a_2} d_{a_2} \\ B_2C_{a_2}x_{a_2} + A_2x_2 + G_2d_2 \\A_{g_2}x_{g_2} + G_{g_2}d_{g_2} \ebm &\in \calS_2\label{eq:S2_invariance} \\[1mm]
        \bbm C_{a_2} & 0 \\ 0 & C_2 \ebm \bbm x_{a_2} \\ x_2 \ebm &= \bbm C^u_{g_2} \\  C^y_{g_2} \ebm x_{g_2}.\label{eq:S2_equalities}
    \end{align}
    We have that $C_1x_1 = C_{a_2}x_{a_2}$ due to \eqref{eq:S12_equalities}, hence \eqref{eq:S2_invariance} yields
    \begin{equation}\label{eq:S2_invariance_modified}
        \bbm A_{a_2}x_{a_2} + G_{a_2} d_{a_2} \\ B_2C_1x_1 + A_2x_2 + G_2d_2 \\A_{g_2}x_{g_2} + G_{g_2}d_{g_2} \ebm \in \calS_2.
    \end{equation}
    Then \eqref{eq:S1_invariance}, \eqref{eq:S12_invariance} and \eqref{eq:S2_invariance_modified} imply that
    \begin{equation}\label{eq:S_invariance}
        \bbm A_{a_1}x_{a_1} + G_{a_1} d_{a_1} \\ B_1C_{a_1}x_{a_1} + A_1x_1 + G_1d_1 \\ B_2C_1x_1 + A_2x_2 + G_2d_2 \\ A_{g_1}x_{g_1} + G_{g_1}d_{g_1} \\ A_{g_2}x_{g_2} + G_{g_2}d_{g_2} \ebm \in \calS
    \end{equation}
    On the other hand, \eqref{eq:S1_equalities} and \eqref{eq:S2_equalities} imply that
    \begin{equation}\label{eq:S_equalities}
        \bbm C_{a_1} & 0 & 0 \\ 0 & 0 & C_2 \ebm \bbm x_{a_1} \\ x_1 \\ x_2 \ebm = \bbm C^u_{g_1} & 0 \\ 0 & C^{y}_{g_2} \ebm \bbm x_{g_1} \\ x_{g_1} \ebm,
    \end{equation}
    Note that $\pi_{\X{a_1}\times\X{1}\times\X{2}}(\calS) = \calV_{a_1}\times\X{1}\times\X{2}$ because $\calS_1$, $\calR$ and $\calS_2$ are full simulation relations. Therefore, we only need to show that $\pi_{\X{g_1}\times\X{g_2}}(\calS) \subset \calV_{g_{1}\to g_2}$, where $\calV_{g_1\to g_2}$ is the consistent subspace of $\gar_1\to\gar_2$. To do this, note that $\pi_{\X{g_1}\times\X{g_2}}(\calS)\subset \calV_{g_1}\times\calV_{g_2}$ because $\pi_{\X{g_1}}(\calS_1) \subset \calV_{g_1}$ and $\pi_{\X{g_2}}(\calS_2) \subset \calV_{g_2}$. In particular, this means that
    \begin{equation}
        \pi_{\X{g_1}\times\X{g_2}}(\calS) \subset\ker\bbm H_1 & 0 \\ 0 & H_2 \ebm.
    \end{equation}
    On the other hand, \eqref{eq:S1_equalities}, \eqref{eq:S12_equalities} and \eqref{eq:S2_equalities} imply that
    \begin{equation}
        C^y_{g_1}x_{g_1} = C_1x_1 = C_{a_2}x_{a_2} = C^u_{g_2}x_{g_2},
    \end{equation}
    for all $(x_{g_1},x_{g_2})\in \pi_{\X{g_1}\times\X{g_2}}(\calS)$. This shows that
    \begin{equation}
        \pi_{\X{g_1}\times\X{g_2}}(\calS) \subset \ker \bbm H_1 & 0 \\ 0 & H_2 \\ C^y_{g_1} & -C^u_{g_2} \ebm
    \end{equation}
    and thus $\pi_{\X{g_1}\times\X{g_2}}(\calS) = \calV_{g_{1}\to g_2}$ due to Remark~\ref{rem:simulation_relation}. Therefore, due to Proposition~\ref{prop:simulation_relation}, \eqref{eq:S_invariance} and \eqref{eq:S_equalities}, it follows that $\calS$ is indeed a full simulation relation of $\ass_1\meet(\sys_1\to\sys_2)$ by $\gar_1\to\gar_2$ and thus the third condition is also satisfied.\endproof
    
    \bibliographystyle{ieeetr}
    \bibliography{../../../references/all}
\end{document}